\documentclass[12pt]{amsart}
\usepackage{tikz}
\usepackage{amssymb}
\allowdisplaybreaks

\def\Wnodal{W_{\scriptstyle\rm\!nodal}}
\def\P{\mathbb{P}}
\def\Pic{\mathop{\rm Pic}\nolimits}
\def\Aut{\mathop{\rm Aut}\nolimits}
\def\nef{\mathop{\rm nef}\nolimits}
\def\Orth{{\rm O}}

\def\Oup{\Orth^{\uparrow}}
\def\xbar{{\bar{x}}}
\def\nubar{{\bar{\nu}}}
\def\rhobar{{\bar{\rho}}}
\def\Wbar{\overline{W}}
\def\F{\mathbb{F}}     
\def\Z{\mathbb{Z}}     
\def\R{\mathbb{R}}     
\def\Atilde{\widetilde{A}}
\def\Dtilde{\widetilde{D}}
\def\Etilde{\widetilde{E}}
\def\tensor{\otimes}
\def\iso{\cong}
\def\sset{\subseteq}
\def\generatedby#1{\langle#1\rangle}
\def\gend#1{\generatedby{#1}}
\def\spanof#1{\generatedby{#1}}
\def\set#1#2{\{#1\mid#2\}}
\def\bigset#1#2{\bigl\{#1\bigm|#2\bigr\}}
\def\vertexouterradius{.2}
\def\vertexinnerradius{.15}
\def\vertex(#1){%
\fill (#1) circle (\vertexouterradius);%
\fill[white] (#1) circle (\vertexinnerradius)}
\def\mywidth{1}
\def\heavywidth{2.5}

\newtheorem{theorem}{Theorem}
\numberwithin{theorem}{section}
\numberwithin{equation}{section}
\numberwithin{figure}{section}
\newtheorem{lemma}[theorem]{Lemma}
\theoremstyle{remark}
\newtheorem*{remarks}{Remarks}

\begin{document}

\title[Congruence subgroups and Enriques surfaces]{Congruence subgroups and Enriques surface automorphisms}
\author{Daniel Allcock}
\thanks{Supported by NSF grant DMS-1101566}
\address{Department of Mathematics\\University of Texas, Austin}
\email{allcock@math.utexas.edu}
\urladdr{http://www.math.utexas.edu/\textasciitilde allcock}
\date{February 6, 2018}

\begin{abstract}
  We give conceptual proofs of some results on the automorphism group
  of an Enriques surface $X$, for which only computational proofs have
  been available.  Namely, there is an obvious upper bound on the
  image of $\Aut X$ in the isometry group of $X$'s numerical lattice,
  and we establish a lower bound for the image that is quite close to
  this upper bound.  These results apply over any algebraically closed
  field, provided that $X$ lacks nodal curves, or that all its nodal
  curves are (numerically) congruent to each other mod~$2$.  In this
  generality these results were originally proven by Looijenga and
  Cossec--Dolgachev, developing earlier work of Coble.
\end{abstract}

\maketitle

\section{Introduction}
\label{sec-introduction}

\noindent
Our goal in this paper is to give conceptual proofs of some known
computer-based results on the group of automorphisms of an Enriques
surface $X$.  These results are valid over any algebraically closed
field.  Of course, $\Aut X$ acts on $\Pic X$, hence on the quotient
$\Lambda$ of $\Pic X$ by its torsion subgroup~$\Z/2$.  This quotient
$\Lambda$ is called the numerical lattice, and is a copy of the famous
$E_{10}$ lattice.  One can describe it as $E_8\oplus U$ where we take
$E_8$ to be negative definite and
$U=\bigl(\begin{smallmatrix}0&1\\1&0\end{smallmatrix}\bigr)$.

  The main object of interest in this paper is the image $\Gamma$ of
  $\Aut X$ in $\Orth(\Lambda)$.  This is ``most'' of $\Aut X$, because
  the kernel of $\Aut X\to\Gamma$ is finite, and in fact very tightly
  constrained \cite[\S7.2]{CD-new}.  All of our arguments concern
  $\Lambda$ and various Coxeter groups acting on it.  For the
  underlying algebraic geometry we refer to \cite{CD-old},
  \cite{CD-new} and
  \cite{Dolgachev-introduction-to-Enriques-surfaces}.

Because $\Lambda$ has signature $(1,9)$, the positive norm vectors in
$\Lambda\tensor\R$ fall into two components.  Just one of these
contains ample classes; we call it the future cone and write
$\Oup(\Lambda)$ for the subgroup of $\Orth(\Lambda)$ preserving it.
Vinberg showed (theorem~\ref{thm-Vinberg-237} below) that
$\Oup(\Lambda)$ is the Coxeter group $W_{237}$ with diagram
\begin{equation}
\label{eq-E10-diagram}
\begin{tikzpicture}[line width=\mywidth, scale=.5]
\draw (0,0) -- (8,0);
\draw (2,0) -- (2,1);
\vertex (0,0);
\vertex (1,0);
\vertex (2,0);
\vertex (3,0);
\vertex (4,0);
\vertex (5,0);
\vertex (6,0);
\vertex (7,0);
\vertex (8,0);
\vertex (2,1);
\end{tikzpicture}%
\end{equation}
Besides preserving $\Lambda$, the main constraint on
$\Gamma$ is that it  must preserve the ample cone, hence its closure, the
numerically effective cone $\nef(X)$.  The nef cone is described in
terms of $X$'s nodal curves (i.e., smooth rational curves), which have
self-inter\-sec\-tion~$-2$ by the adjunction formula.  If $X$ has
nodal curves 
then $\nef(X)$ consists of the vectors in $\Lambda\tensor\R$ having
nonnegative inner product with all of them.  In the special case that
$X$ lacks nodal curves, $\nef(X)$ is the closure of the future cone.

The remaining constraint on $\Gamma$ concerns the $\F_2$ vector space
$V:=\Lambda/2\Lambda$.  Dividing lattice vectors' norms by~$2$ and
then reducing modulo~$2$ defines on $V$ an $\F_2$ quadratic form of
plus type.  ``Plus type'' means that $V$ has totally isotropic spaces
of largest possible dimension, in this case~$5$.  Although we won't use this
property, it does explain the presence of some superscripts $+$.  We
write $\Orth(V)$ for the isometry group of this quadratic form.  We will use ATLAS
notation for group structures and finite groups throughout the paper;
see \cite{ATLAS}, especially \S5.2.  In this notation, $\Orth(V)$ has
structure $\Orth_{10}^+(2) :2$.  (Caution: the ``$\Orth$'' in the ATLAS notation
$\Orth_{10}^+(2)$ indicates
the simple composition factor of the orthogonal group---in this
case an 
index~$2$ subgroup. Some authors write $\Orth_{10}^+(2)$ for $\Orth(V)$ itself.)

\begin{theorem}[The unnodal case]
\label{thm-unnodal-case}
Suppose an Enriques surface $X$ has no nodal curves.  Then
$\Gamma$ contains the level two congruence subgroup $W_{237}(2)$,
meaning the kernel of the natural map $\Oup(\Lambda)\to\Orth(V)$.
\end{theorem}

So $\Gamma$ must be one of the finitely many groups between
$W_{237}(2)$ and $W_{237}$.  Because $\Oup(\Lambda)\to\Orth(V)$ is a
surjection, the possibilities correspond to subgroups of
$\Orth_{10}^+(2):2$.  Different $X$ can lead to different $\Gamma$, so
one cannot say much more without specifying $X$ more closely.  
If $X$ is unnodal then $\Aut X$ acts faithfully on $\Lambda$, by
\cite[Thm.\ 7.3.6]{CD-new}.  So one can identify $\Gamma$ with $\Aut X$.

In characteristic~$0$ one can describe $\Gamma$ in terms of the period
of the K3 surface which covers $X$.  In this way one can show that for
a generic Enriques surface without nodal curves, $\Gamma$ is exactly
$W_{237}(2)$; see \cite{Barth-Peters}.  The positive characteristic
analogue of this seems to be open.

\medskip
Given a nodal curve, regarded as an element of $\Pic X$, the
corresponding {\it nodal root} means its image under $\Pic
X\to\Lambda$.  It is called a root because it has norm~$-2$ and so the
reflection in it is an isometry of $\Lambda$.  Distinct nodal curves
have intersection number${}\geq0$, hence distinct images in
$\Lambda$.  So the nodal curves and nodal roots are in natural
bijection.

Given a nodal root, its corresponding {\it nodality class} means its
image in $V$, always an anisotropic vector.  Theorem~\ref{thm-1-nodal-case} below is
the analogue of theorem~\ref{thm-unnodal-case} in the ``$1$-nodality-class case'': when $X$ has
at least one nodal curve, and all nodal curves represent a single
nodality class.
We will use lowercase letters with bars to
indicate elements of $V$, whether or not we have in mind particular
lifts of them to $\Lambda$.  By definition of the quadratic form on
$V$, every nodality class $\nubar$ is anisotropic.  So its transvection
$\xbar\mapsto\xbar+(\xbar\cdot\nubar)\nubar$ is an isometry of $V$.
(The inner product on $V$ is defined by reducing inner products in $\Lambda$ mod~$2$,
	and carries strictly less information than the quadratic form.  Although the
	transvection in any element of $V$ preserves the inner product, only
	transvections in anisotropic vectors preserve the quadratic form.)

We indicate stabilizers using subscripts, for example
$\Oup(\Lambda)_\nubar$ in the next theorem.

\begin{theorem}[The $1$-nodality-class case]
\label{thm-1-nodal-case}
Suppose an Enriques surface $X$ has a single nodality class
$\nubar\in V$.  Then
the $\Oup(\Lambda)$-stabilizer $\Oup(\Lambda)_\nubar$ of $\nubar$ is the Coxeter group
\begin{equation}
\label{eq-big-diagram-for-1-nodal-case}
\begin{tikzpicture}[line width=\mywidth, scale=.5]
\draw (0,0) -- (8,0);
\draw[line width=\heavywidth](8,0) -- (9,0);
\draw (3,0) -- (3,1);
\vertex (0,0);
\vertex (1,0);
\vertex (2,0);
\vertex (3,0);
\vertex (4,0);
\vertex (5,0);
\vertex (6,0);
\vertex (7,0);
\vertex (8,0);
\vertex (9,0);
\vertex (3,1);
\end{tikzpicture}%
\end{equation}
Write $W_{246}$ for the subgroup generated by the reflections
corresponding to the leftmost~$10$ nodes.  
Then
\begin{enumerate}
\item
  \label{item-nef-is-union-of-W246-translates}
$\nef(X)$ is the union of
the $W_{246}$-translates of the fundamental chamber of the Coxeter group
\eqref{eq-big-diagram-for-1-nodal-case}.
\item
  \label{item-W246-is-full-stabilizer}
$\Gamma$ lies in $W_{246}$, which is the full
$\Oup(\Lambda)$-stabilizer of $\nef(X)$.
\item
  \label{item-Gamma-contains-W-bar-246(2)}
$\Gamma$ contains the subgroup $\Wbar_{\!246}(2)$ defined as the
  subgroup of $W_{246}$ that acts trivially on $\nubar^\perp\sset V$.
\item
  \label{item-transitivity-on-facets}
  $\Wbar_{\!246}(2)$ acts transitively on the facets of $\nef(X)$.
\item
  \label{item-transitivity-on-nodal-curves}
$\Aut X$ acts transitively on the nodal curves of~$X$.
\end{enumerate}
\end{theorem}

\begin{remarks}  
(a) The heavy edge in the diagram
  indicates parallelism of the corresponding hyperplanes in $H^9$, or
  equivalently that the last pair of roots has intersection number~$2$.  

  (b) Suppose $X$ is a generic nodal Enriques surface.  Then $\Aut X$ acts
  faithfully on $\Lambda$, so  it can be identified with
  $\Gamma$; see \cite[Prop.\ 7.4.1]{CD-new}.  Furthermore, in characteristic${}\neq2,3,5,7$ or $17$,
  $\Gamma$ coincides with $\Wbar_{\!246}(2)$, by
  \cite[Thm.\ 1]{Cossec-Dolgachev-published}.

  (c) $\Wbar_{\!246}(2)$ is the group called $\Wbar(2)$ by Cossec and
  Dolgachev \cite{Cossec-Dolgachev-published}.  But, contrary to what
  the notation might suggest, the kernel of our $W_{246}\to\Orth(V)$ is not
  the same as their $W(2)$.  This is because they define their
  congruence subgroups with respect to the Reye lattice rather than
  the Enriques lattice~$\Lambda$.  The Reye lattice has index~$2$ in
  $\Lambda$: it is the preimage of $\nubar^\perp\sset V$.
  %
\end{remarks}

Theorems \ref{thm-unnodal-case} and~\ref{thm-1-nodal-case} are modern forms of results of Coble
\cite[Thms.\ (4) and (30)]{Coble}. But Cossec--Dolgachev
\cite[p.~162]{CD-old} state that his proofs were incorrect.  They
credit Looijenga with the first proof of theorem~\ref{thm-unnodal-case}, never
published, and give  proofs of both theorems, following Looijenga's
ideas.  See
\cite[Thms.\ 2.10.1 and~2.10.2]{CD-old}.  Their proof of the first
relied on a lengthy hand computation, and the second required computer
assistance.

The author is grateful to RIMS (Kyoto University) for its hospitality
while working on this paper, to Igor Dolgachev for posing the problem
of improving on the computer computations in \cite{CD-old}, and to
Shigeru Mukai for stimulating discussions.

\section{The case of no nodal curves}
\label{sec-no-nodal}

\noindent
A root means a lattice vector of norm~$-2$.
In this section our model for $\Lambda$ is the span of the roots in
figure~\ref{fig-simple-roots-for-E10}.  Two
of them have inner product $1$ or~$0$, according to whether they are
joined or not.  By the theory of reflection groups \cite[V.4]{Bourbaki}, these $10$ vectors
form a set of simple roots for the group $W_{237}$ generated by their
reflections.  We will write $W_{235}$ resp.\ $W_{236}$ for
the subgroup generated by the reflections corresponding to the
top~$8$ resp.\ $9$ nodes.  Also, we will write $\Lambda_0$ for the
span of the first $8$ roots.  This is a copy of the $E_8$ lattice in
the ``odd'' coordinate system, namely
$$
\bigset{(x_1,\dots,x_8)}{\hbox{all $x_i$ in $\Z$ or all in $\Z+\frac12$, and
    $\sum x_i\equiv 2x_8$ mod~$2$}}
$$
from \cite[\S8.1 of \hbox{Ch.\ 4}]{Conway-Sloane}.  Its isometry
group is the $E_8$ Weyl group $W_{235}$, which has structure
$2\cdot\Orth_8^+(2):2$.  Sometimes we will write lattice vectors as
$(x;y,z)$ with $x\in\Lambda_0$ and $y,z\in\Z$, and inner product
$(x;y,z)\cdot(x';y',z')=x\cdot x'+y z' + y' z$.

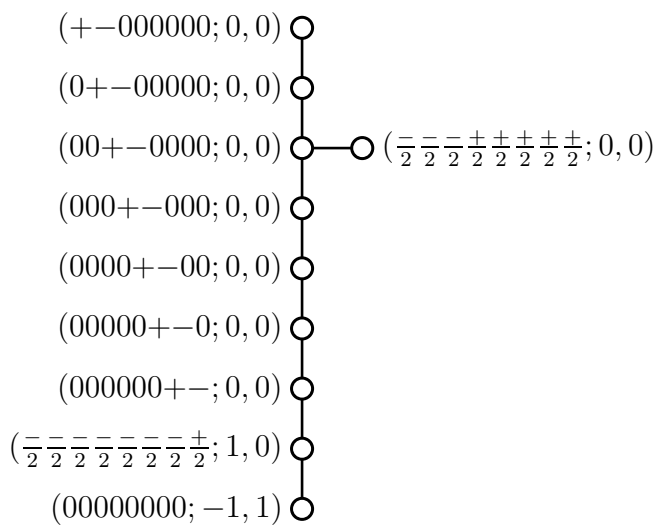
\begin{figure}
\begin{tikzpicture}[line width=\mywidth, scale=.8]
\draw (0,0) -- (0,8);
\draw (0,6) -- (1,6);
\vertex (0,0);
\vertex (0,1);
\vertex (0,2);
\vertex (0,3);
\vertex (0,4);
\vertex (0,5);
\vertex (0,6);
\vertex (0,7);
\vertex (0,8);
\vertex (1,6);
\draw (0,0) node {\llap{$(00000000;-1,1)$}\kern15pt};
\draw (1,6) node {\kern15pt\rlap{$\textstyle(\frac-2\frac-2\frac-2\frac+2\frac+2\frac+2\frac+2\frac+2;0,0)$}};
\draw (0,2) node {\llap{$(000000{+}{-};0,0)$}\kern15pt};
\draw (0,3) node {\llap{$(00000{+}{-}0;0,0)$}\kern15pt};
\draw (0,4) node {\llap{$(0000{+}{-}00;0,0)$}\kern15pt};
\draw (0,5) node {\llap{$(000{+}{-}000;0,0)$}\kern15pt};
\draw (0,6) node {\llap{$(00{+}{-}0000;0,0)$}\kern15pt};
\draw (0,7) node {\llap{$(0{+}{-}00000;0,0)$}\kern15pt};
\draw (0,8) node {\llap{$({+}{-}000000;0,0)$}\kern15pt};
\draw (0,1) node {\llap{$\textstyle(\frac-2\frac-2\frac-2\frac-2\frac-2\frac-2\frac-2\frac+2;1,0)$}\kern15pt};
\end{tikzpicture}%
\caption{Simple roots for $W_{237}=\Oup(\Lambda)$, with respect to
  the norm $(x_1,\dots,x_8;y,z)^2=-x_1^2-\cdots-x_8^2+2y z$.  We have
  abbreviated $\pm1$ to $\pm$ and hidden some commas.}
\label{fig-simple-roots-for-E10}
\end{figure}

\begin{lemma}[The stabilizer of a null vector]
\label{lem-null-vector-stabilizer-237-case}
  $W_{236}$ is the full stabilizer $\Orth(\Lambda)_\rho$ of
  the null vector $\rho=(0;1,0)$.  It has structure $\Lambda_0:W_{235}$,
  where $\Lambda_0$ indicates the group of ``translations''
\begin{align}
\notag
&(x;0,0)\mapsto(x;-\lambda\cdot x,0)\\
\label{eq-translations-in-E9}
T_{\lambda\in\Lambda_0}:{}&(0;1,0)\mapsto(0;1,0)\\
\notag
&(0;0,1)\mapsto(\lambda;-\lambda^2/2,1)
\end{align}
\end{lemma}

\begin{proof}
The $T_\lambda$ are called translations because of how they act on
hyperbolic space when $\rho$ is placed at infinity in the upper
halfspace model.  One checks that they are isometries, that
$T_{\lambda+\mu}=T_\lambda T_\mu$, and that $W_{235}=\Aut\Lambda_0$
acts on them in the same way it acts on $\Lambda_0$.
Next, $W_{236}$ contains the reflection in
$\lambda=(\frac-2\frac-2\frac-2\frac-2\frac-2\frac-2\frac-2\frac+2;00)$,
because this is a root of $\Lambda_0$.  Also, $W_{236}$ contains the
reflection in
$(\frac-2\frac-2\frac-2\frac-2\frac-2\frac-2\frac-2\frac+2;10)$,
because this root is second from the bottom in figure~\ref{fig-simple-roots-for-E10}.  The product of
these two reflections is $T_{\pm\lambda}$, the sign depending on the
order of the factors.  So $W_{236}$ contains the translation by a
root of $\Lambda_0$.  Conjugating by $W_{235}$ shows that $W_{236}$
contains the translations by all the roots of $\Lambda_0$.  Since  $\Lambda_0$ is spanned by its roots, 
$W_{236}$ contains all translations.

The translations act transitively on
$\set{(x;-x^2/2,1)}{x\in\Lambda_0}$, which is the set of null vectors
having inner product~$1$ with $\rho$.
The simultaneous stabilizer of
$\rho$ and  $(0;0,1)$ is the orthogonal group of $\Lambda_0$,
which is $W_{235}\sset W_{236}$.  Since 
$\Oup(\Lambda)_\rho$  and its subgroup $W_{236}$
act transitively
on the same set, with the same stabilizer, they are the same group.
\end{proof}

The proof of the following theorem of Vinberg illustrates the
technique of cusp-counting, which we will use several times.  To avoid
repetition we take ``null vector'' to mean a future-directed
primitive lattice vector of norm~$0$.  ``Cusp counting'' means: when a
Coxeter group acts on an integer quadratic form of signature $(1,n)$
and has
finite volume fundamental chamber in hyperbolic space, then its orbits on
null vectors are in bijection with the ideal vertices of the chamber.
And these in turn are in bijection with the maximal affine subdiagrams
of the Coxeter diagram.

\begin{theorem}[Vinberg \cite{Vinberg-E10}]
\label{thm-Vinberg-237}
  $\Oup(\Lambda)=W_{237}$%
.
\end{theorem}

\begin{proof}
The image in hyperbolic $9$-space of the fundamental chamber has
finite volume, with all vertices in $H^9$ except for one on its
boundary.  This follows from the general theory of hyperbolic
reflection groups: the vertices in $H^9$ correspond to the rank~$9$
spherical subdiagrams of figure~\ref{fig-simple-roots-for-E10}, and the last vertex
corresponds to the affine  subdiagram $\Etilde_8$.  It follows that there is
only one $W_{237}$-orbit of null vectors, i.e., $W_{237}$ acts
transitively on them.  Since $W_{237}$ contains the full
$\Oup(\Lambda)$-stabilizer of one of them (lemma~\ref{lem-null-vector-stabilizer-237-case}), it is all of
$\Oup(\Lambda)$.
\end{proof}

The most important ingredient in the proof of
theorem~\ref{thm-unnodal-case} is the construction of automorphisms of
$X$, for which we refer to the proof of theorem~$3$ in
\cite[\S6]{Dolgachev-introduction-to-Enriques-surfaces}.  $\Lambda$
has many direct sum decompositions as a copy of $E_8$ plus a copy
of~$U$.  For every such decomposition, the transformation which
negates the $E_8$ summand is called a Bertini involution, and arises
from an automorphism of $X$.
(Very briefly: consider the linear system $|2E_1+2E_2|$, where $E_1$
and $E_2$ are the effective classes corresponding to the null vectors
in the $U$ summand.  This is a $2$-to-$1$ map onto a $4$-nodal quartic
del Pezzo surface in $\P^4$, and the Bertini involution is the deck
transformation of this covering.)
Bertini involutions obviously lie in the
level~$2$ congruence subgroup of $\Oup(\Lambda)$, hence in
$W_{237}(2)$.  Also, every conjugate of a Bertini involution is again
an Bertini involution.  So the group they generate is normal in
$\Orth(\Lambda)$.

\begin{proof}[Proof of theorem~\ref{thm-unnodal-case}]
The proof amounts to showing that the Bertini involutions generate
$W_{237}(2)$.  We write $S$ (``small'') for the group they generate,
and think of $\Oup(\Lambda)$ as the ``large'' group.  To understand
the relation between small and large, we will introduce a ``medium''
group $M$.  Its relationships with $S$ and $\Oup(\Lambda)$ are easy to
work out. Then the relationship between $S$ and $\Oup(\Lambda)$ will
be visible.

We define $M$ as the group generated by $S$ and $W_{236}$.  The
central involution $B$ of $W_{235}\sset M$ is a Bertini involution.
Also, its conjugacy action on $\Lambda_0\sset W_{236}$ is inversion.
Being normal, $S$ contains
$$
T_\lambda B T_{\lambda}^{-1}\circ B^{-1}
=
T_\lambda \circ B T_{\lambda}^{-1} B^{-1}
=
T_\lambda\circ T_{-\lambda}^{-1}
=
T_{2\lambda}
$$
for all $\lambda\in\Lambda_0$.  It follows that
$M/S$ is a quotient of $W_{236}/\gend{B,{\rm
    all\ }T_{2\lambda}}=2^8:\Orth_8^+(2):2$.
On the other hand, it is easy to see that $W_{236}$ acts on $V$ as the full
$\Orth(V)$-stabilizer of $\rhobar$, which has structure
$2^8:\Orth_8^+(2):2$.  (Repeat the proof of lemma~\ref{lem-null-vector-stabilizer-237-case}, reduced mod~$2$.)
We have shown that $S$  has index${}\leq|2^8:\Orth_8^+(2)\cdot2|$ in $M$, and lies in the
kernel of the surjection
$M\to \Orth(V)_\rhobar\iso 2^8:\Orth_8^+(2)\cdot2$.
So $S$ coincides with
the kernel.  That is, $M\to\Orth(V)$ induces an
isomorphism $M/S\to\Orth(V)_\rhobar$.

The advantage of working with $M$ rather than $S$ is that 
it contains the  Coxeter group $M_0$ whose simple roots are shown in
figure~\ref{fig-simple-roots-for-subgroup-no-nodal-case}.  We will
see later that in fact $M_0$ is all of $M$; for now we just prove
$M_0\sset M$.  First,
$M$ contains
$W_{236}$ by definition.  To see that $M$ contains the reflection in
the last root (the lower right one), note that
$r=(\frac-2\frac-2\frac-2\frac-2\frac-2\frac-2\frac-2\frac+2;0,0)$ is a
root of $\Lambda_0$, so its reflection lies in $W_{235}$.  Choose an
element $\lambda$ of $\Lambda_0$ having inner product~$-1$
with it.  Then $T_{2\lambda}\in S$ sends $r$ to 
$(\frac-2\frac-2\frac-2\frac-2\frac-2\frac-2\frac-2\frac+2;2,0)$.
Now consider the conjugate of $T_{2\lambda}$ by the isometry of
$\Lambda$ which exchanges the last two coordinates.  This lies in $S$
by normality, and  sends $r$ to 
$(\frac-2\frac-2\frac-2\frac-2\frac-2\frac-2\frac-2\frac+2;0,2)$.
Therefore $M$ contains the reflection in this root.
This finishes the proof
that $M$ contains $M_0$.

\begin{figure}
\begin{tikzpicture}[line width=\mywidth, scale=.8]
\draw (0,1) -- (0,8);
\draw (0,6) -- (1,6);
\draw (0,2) -- (1,2);
\vertex (1,6);
\vertex (0,2);
\vertex (0,3);
\vertex (0,4);
\vertex (0,5);
\vertex (0,6);
\vertex (0,7);
\vertex (0,8);
\vertex (0,1);
\vertex (1,2);
\draw (1,6) node {\kern15pt\rlap{$\textstyle(\frac-2\frac-2\frac-2\frac+2\frac+2\frac+2\frac+2\frac+2;00)$}};
\draw (0,2) node {\llap{$(000000{+}{-};00)$}\kern15pt};
\draw (0,3) node {\llap{$(00000{+}{-}0;00)$}\kern15pt};
\draw (0,4) node {\llap{$(0000{+}{-}00;00)$}\kern15pt};
\draw (0,5) node {\llap{$(000{+}{-}000;00)$}\kern15pt};
\draw (0,6) node {\llap{$(00{+}{-}0000;00)$}\kern15pt};
\draw (0,7) node {\llap{$(0{+}{-}00000;00)$}\kern15pt};
\draw (0,8) node {\llap{$({+}{-}000000;00)$}\kern15pt};
\draw (0,1) node {\llap{$\textstyle(\frac-2\frac-2\frac-2\frac-2\frac-2\frac-2\frac-2\frac+2;10)$}\kern15pt};
\draw (1,2) node {\kern15pt\rlap{$(\frac-2\frac-2\frac-2\frac-2\frac-2\frac-2\frac-2\frac+2;02)$}};
\end{tikzpicture}%
\caption{Simple roots for  $M=\Oup(\Lambda)_\rhobar$; see the proof of theorem~\ref{thm-Vinberg-237}.}
\label{fig-simple-roots-for-subgroup-no-nodal-case}
\end{figure}

Next we claim that $M_0$ is all of $\Oup(\Lambda)_{\rhobar}$.  The
fact that $\Oup(\Lambda)_\rhobar$ contains $M_0$ (even $M$) is
obvious.  Now observe that $M_0$'s chamber has finite volume, with~$3$
cusps, corresponding to the $\Dtilde_8$ subdiagram and two
$\Etilde_8$ subdiagrams.  Therefore $M_0$ has~$3$ orbits on null
vectors.  On the other hand, $\Oup(\Lambda)_\rhobar$ has at least~$3$ orbits on null vectors, since it has three orbits on
isotropic vectors in $V$.  (Namely: $\rhobar$ itself, the other isotropic
vectors orthogonal to $\rhobar$, and the
isotropic vectors not orthogonal to $\rhobar$.)  So 
$\Oup(\Lambda)_\rhobar$
and its subgroup $M_0$
have the same orbits on null vectors. The stabilizer of $\rho$ in
either of them is $W_{236}$,
proving $M_0=\Oup(\Lambda)_\rhobar$.  Since $M_0\sset
M\sset\Oup(\Lambda)_\rhobar$, this also shows the equality of $M$ with
these groups.

Finally, we have
\begin{align*}
[\Oup(\Lambda):S]&{}=[\Oup(\Lambda):M][M:S]
\\
&{}=\bigl[\Oup(\Lambda):\Oup(\Lambda)_\rhobar\bigr]\,\bigl|\Orth(V)_\rhobar\bigr|
\\
&{}=\bigl[\Orth(V):\Orth(V)_\rhobar\bigr]\,\bigl|\Orth(V)_\rhobar\bigr|
\\
&{}=\bigl|\Orth(V)\bigr|
\end{align*}
Since $S$ lies in the kernel of the surjection
$\Oup(\Lambda)\to\Orth(V)$, it must be the whole kernel, finishing the proof.
\end{proof}

\section{Preparation for the $1$-nodality-class case}
\label{sec-preparation-for-1-nodal}

\noindent
This section can be summarized as ``the same as section~\ref{sec-no-nodal} with
$E_7\oplus U$ in place of $E_8\oplus U$''.  To tighten the analogy it
is necessary to use the ``even'' coordinate system for the $E_8$
lattice in place of the ``odd'' one we used in the previous section.
So now we take the $E_8$ lattice to consist of the vectors
$(x_1,\dots,x_8)$ with even coordinate sum and either all entries in
$\Z$ or all in $\Z+\frac12$.  See \cite[\S8.1 of
  Ch.\ 4]{Conway-Sloane}; these coordinates differ from those of
section~\ref{sec-no-nodal} by
negating any coordinate.  We take $\Lambda$ to consist of the vectors
$(x_1,\dots,x_8;y,z)$ with $(x_1,\dots,x_8)$ in the $E_8$ lattice and
$y,z\in\Z$.  The norm is still $-x_1^2-\cdots-x_8^2+2y
z$.  Mimicking our notation from section~\ref{sec-no-nodal}, we write $\Lambda_0$
for the sublattice $\{(x_1,\dots,x_8;0,0)\}$ of $\Lambda$.

We write $\nu$ for the root
$(\frac+2\frac+2\frac+2\frac+2\frac+2\frac+2\frac+2\frac+2;00)$ of
$\Lambda_0$.
It stands for ``nodal root'', although for this section it is just a root.
Its orthogonal complement in $\Lambda_0$ is a copy of
the $E_7$ root lattice, and its full orthogonal complement $\nu^\perp$
in $\Lambda$ is $E_7\oplus U$.  It is easy to see that $\nu^\perp$ is
spanned by the roots in figure~\ref{fig-simple-roots-245}, and that their inner products
are indicated in the usual way by the edges of the diagram.  In
particular they form a set of simple roots for the Coxeter group
$W_{245}$ generated by their reflections.  We also write $W_{244}$
for the subgroup generated by the reflections in the top~$8$ roots,
and regard both these groups as acting on all of $\Lambda$, not just
$\nu^\perp=E_7\oplus U$.  The next
two results are proven the same way lemma~\ref{lem-null-vector-stabilizer-237-case} and theorem~\ref{thm-Vinberg-237}
were.

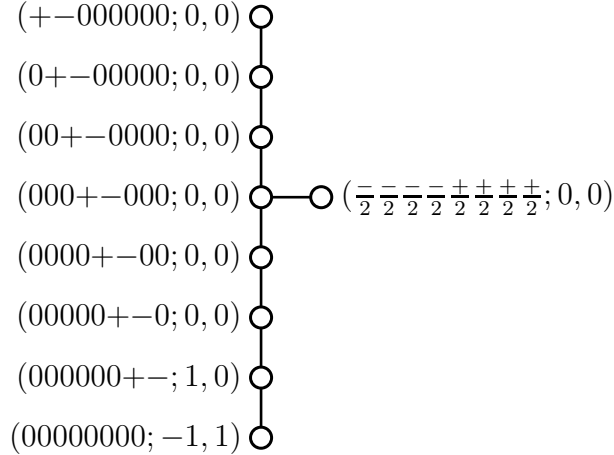
\begin{figure}
\begin{tikzpicture}[line width=\mywidth, scale=.8]
\draw (0,-1) -- (0,6);
\draw (0,3) -- (1,3);
\vertex (0,-1);
\vertex (0,0);
\vertex (0,1);
\vertex (0,2);
\vertex (0,3);
\vertex (0,4);
\vertex (0,5);
\vertex (0,6);
\vertex (1,3);
\draw (0,6) node {\llap{$({+}{-}000000;0,0)$}\kern15pt};
\draw (0,5) node {\llap{$(0{+}{-}00000;0,0)$}\kern15pt};
\draw (0,4) node {\llap{$(00{+}{-}0000;0,0)$}\kern15pt};
\draw (0,3) node {\llap{$(000{+}{-}000;0,0)$}\kern15pt};
\draw (0,2) node {\llap{$(0000{+}{-}00;0,0)$}\kern15pt};
\draw (0,1) node {\llap{$(00000{+}{-}0;0,0)$}\kern15pt};
\draw (0,0) node {\llap{$(000000{+}{-};1,0)$}\kern15pt};
\draw (0,-1) node {\llap{$(00000000;-1,1)$}\kern15pt};
\draw (1,3) node {\kern15pt\rlap{$(\frac-2\frac-2\frac-2\frac-2\frac+2\frac+2\frac+2\frac+2;0,0)$}};
\end{tikzpicture}%
\caption{Simple roots for $\Oup(\Lambda)_\nu$, 
where~$\nu$ is the root
$(\frac+2\frac+2\frac+2\frac+2\frac+2\frac+2\frac+2\frac+2;00)$ of~$\Lambda$;
see theorem~\ref{thm-Vinberg-245}.}
\label{fig-simple-roots-245}
\end{figure}

\begin{lemma}
\label{lem-null-vector-stabilizer-244}
$W_{244}$ is the full stabilizer of the null vector $\rho=(0;1,0)$
in $\Orth(\Lambda)_\nu$.
\qed
\end{lemma}

\begin{theorem}[Vinberg]
\label{thm-Vinberg-245}
$W_{245}$ is all of $\Oup(\Lambda)_\nu$.
\qed
\end{theorem}

Bertini involutions: from the presence of an $E_8$ diagram in
figure~\ref{fig-simple-roots-245} we see that $\nu^\perp$ has sublattices isomorphic to
$E_8$.  Every $E_8$ sublattice is unimodular, hence a direct summand
of $\Lambda$, so the involution that negates the $E_8$ summand is an
isometry.  Since this summand was chosen in $\nu^\perp$, we obtain an
element of $\Oup(\Lambda)_\nu$.  These are called Bertini involutions,
and act trivially on $V$.

Kantor involutions: by construction, $\nu^\perp$ has direct sum
decompositions $E_7\oplus U$.  For any such decomposition, the central
involution in $W(E_7)$ acts on $\Lambda$ by negation on the $E_7$ summand and
trivially on the $U$ summand.  These are called Kantor involutions.
Every one acts on $V$ by the transvection in $\nubar$.  (Proof: the
complement in $\Lambda$ of the $U$ summand is a copy of $E_8$
containing $\spanof{\nu}\oplus E_7$.  The Kantor involution is the product of the
negation map of this $E_8$, which acts trivially on $V$, with the
reflection in $\nu$.)   

\begin{theorem}
\label{thm-Kantor-and-Bertini-inside-245}
  The Kantor and Bertini involutions generate the subgroup
$\Oup(\Lambda)_{\nu,\nubar^\perp}$
  of
$\Oup(\Lambda)$ that fixes $\nu$ and acts trivially on $\nubar^\perp\sset V$.
\end{theorem}

\begin{proof}
We reuse our strategy from theorem~\ref{thm-unnodal-case}.  That is, we write $S$ for
the subgroup of $\Oup(\Lambda)_\nu$ generated by the Kantor and
Bertini involutions, and think of it as ``small''.  We think of
$\Oup(\Lambda)_\nu$ as ``large''.  Obviously $S$ is normal in $\Orth(\Lambda)_\nu$.  To
relate these groups we define the ``medium'' group $M$ to be generated
by $S$ and $W_{244}$.

Recall that $W_{244}$ has structure
$E_7:W(E_7)=E_7:(2\times\Orth_7(2))$ where the initial $E_7$ indicates the root lattice regarded
as a group. The central involution in $2\times\Orth_7(2)$ is a Kantor
involution.  Mimicking the proof of theorem~\ref{thm-unnodal-case} shows that
$M/(S\cap M)$ is a quotient of $2^7:\Orth_7(2)$.  Continuing the mimicry,
the image of $M$ in $\Orth(V)$ has structure $2^7:(2\times\Orth_7(2))$,
which is the simultaneous stabilizer $\Orth(V)_{\nubar,\rhobar}$.  (Note:
$2^7$ and $2\times\Orth_7(2)$ are subgroups of $2^8$ and
$\Orth_8^+(2):2$ from the proof of theorem~\ref{thm-unnodal-case}.
The $2^7$ is
the subgroup of $\Orth(V)_\nubar$ that fixes $\rhobar$ and acts trivially
on $\rhobar^\perp/\gend{\rhobar}$,  and $2\times\Orth_7(2)$ acts faithfully on
$\rhobar^\perp/\gend{\rhobar}$.)  Every Kantor involution acts trivially
on $\nubar^\perp$, and the image of $M$ in $\Orth(\nubar^\perp)\iso\Orth_9(2)$ has
structure $2^7:\Orth_7(2)$.  It follows that $S$ is the kernel of the
action of $M$ on $\nubar^\perp\sset V$.  So we may identify $M/S$ with
the stabilizer of $\rhobar$ in $\Orth(\nubar^\perp)$.

\begin{figure}
\begin{tikzpicture}[line width=\mywidth, scale=.8]
\draw[line width=\heavywidth](0,6) -- (0,7);
\draw (0,0) -- (0,6);
\draw (0,3) -- (1,3);
\draw (0,1) -- (1,1);
\vertex (0,0);
\vertex (0,1);
\vertex (0,2);
\vertex (0,3);
\vertex (0,4);
\vertex (0,5);
\vertex (0,6);
\vertex (0,7);
\vertex (1,1);
\vertex (1,3);
\draw (0,7) node {\llap{$e_{10}=(-\frac32,\frac12\frac12\frac12\frac12\frac12\frac12,-\frac32;12)$}\kern15pt};
\draw (0,6) node {\llap{$e_1=({+}{-}000000;00)$}\kern15pt};
\draw (0,5) node {\llap{$e_2=(0{+}{-}00000;00)$}\kern15pt};
\draw (0,4) node {\llap{$e_3=(00{+}{-}0000;00)$}\kern15pt};
\draw (0,3) node {\llap{$e_4=(000{+}{-}000;00)$}\kern15pt};
\draw (0,2) node {\llap{$e_5=(0000{+}{-}00;00)$}\kern15pt};
\draw (0,1) node {\llap{$e_6=(00000{+}{-}0;00)$}\kern15pt};
\draw (0,0) node {\llap{$e_7=(000000{+}{-};10)$}\kern15pt};
\draw (1,1) node {\kern15pt\rlap{$(000000{+}{-};02)=e_9$}};
\draw (1,3) node {\kern15pt\rlap{$(\frac-2\frac-2\frac-2\frac-2\frac+2\frac+2\frac+2\frac+2;00)=e_8$}};
\end{tikzpicture}%
\caption{Simple roots for $M=\Oup(\Lambda)_{\nu,\rhobar}$
where $\rho=(0;1,0)$ and
$\nu=(\frac+2\frac+2\frac+2\frac+2\frac+2\frac+2\frac+2\frac+2;00)$;
see the proof of theorem~\ref{thm-Kantor-and-Bertini-inside-245}.}
\label{fig-simple-roots-subgroup-of-245}
\end{figure}

Next we claim that $M$ contains the Coxeter group $M_0$ with simple
roots pictured in figure~\ref{fig-simple-roots-subgroup-of-245}.   First, $e_1,\dots,e_8$
are the simple roots of $W_{244}$, whose reflections lie in $M$ by
definition.  The proof that $M$ contains the reflection in $e_9$ is
exactly the same as in the proof of theorem~\ref{thm-unnodal-case}.  (Only the Kantor
involutions are needed.)  For $e_{10}$, observe that it and
$e_2,\dots,e_8$ span a copy of the lattice $A_1\oplus E_7$.  
Furthermore, $(e_{10}+e_5+e_7+e_8)/2$ lies in $\Lambda$, so 
the saturation of this $A_1\oplus E_7$ is a copy of $E_8$.
The reflection in $e_{10}$ is equal to
the Bertini involution of this $E_8$, times the central
involution of the copy of $W(E_7)\sset W_{244}$ generated by the reflections in
$e_2,\dots,e_8$. 
Therefore $M$
contains this reflection.

The same argument as in the proof of theorem~\ref{thm-unnodal-case} shows that
$M_0=M=\Oup(\Lambda)_{\nu,\rhobar}$.  (This time the affine diagrams
are $\Dtilde_6\Atilde_1$ and two $\Etilde_7$'s.)
The final step of the proof is also conceptually the same as before.  Namely,
\begin{align*}
  [\Oup(\Lambda)_\nu:S]
  &{}=
  [\Oup(\Lambda)_\nu:M][M:S]
  \\
  &{}=[\Oup(\Lambda)_\nu:\Oup(\Lambda)_{\nu,\rhobar}]\bigl|\Orth(\nubar^\perp)_{\rhobar}\bigr|
  \\
  &{}=[\Orth(V)_\nubar:\Orth(V)_{\nubar,\rhobar}]\bigl|\Orth(\nubar^\perp)_{\rhobar}\bigr|
  \\
  &{}=[\Orth(\nubar^\perp):\Orth(\nubar^\perp)_{\rhobar}]\bigl|\Orth(\nubar^\perp)_{\rhobar}\bigr|
  \\
  &{}=|\Orth(\nubar^\perp)|.
\end{align*}
From this and the fact that $S$ acts trivially on $\nubar^\perp\sset
V$, it follows that $S$ is the full kernel  of
$\Oup(\Lambda)_\nu$'s action  on $\nubar^\perp$.
\end{proof}

\section{The $1$-nodality-class case}
\label{sec-1-nodal}

\noindent
In this section we continue to use the previous section's model for~$\Lambda$.  We suppose $X$ is an Enriques surface with a single nodality
class $\nubar\in V=\Lambda/2\Lambda$, and we fix some nodal root
$\nu\in\Lambda$ lying over it.

All roots of $\Lambda$ are equivalent under isometries (since
figure~\ref{fig-simple-roots-for-E10} is simply laced). So we may choose the identification
between $\Lambda$ and $X$'s numerical lattice such that $\nu$ is any
chosen root.  We choose
$\nu=(\frac+2\frac+2\frac+2\frac+2\frac+2\frac+2\frac+2\frac+2;0,0)$,
which is compatible with the previous section's notation.  To the simple roots from
figure~\ref{fig-simple-roots-245} we may adjoin two more roots, to
obtain simple roots for the larger Coxeter group  whose diagram appears
in in figure~\ref{fig-simple-roots-for-nu-bar-stabilizer}.  The extra simple roots are $\nu$ and $e_{10}'$.
We write $C$ for the fundamental chamber for these $11$ simple roots.
We define $W_{246}$ as the group generated by the reflections in the
top~$10$ roots, and continue writing $W_{244}$ and $W_{245}$ as
before.

\begin{figure}
\begin{tikzpicture}[line width=\mywidth, scale=.8]
\draw[line width=\heavywidth](0,-3) -- (0,-2);
\draw (0,-2) -- (0,6);
\draw (0,3) -- (1,3);
\vertex (0,-3);
\vertex (0,-2);
\vertex (0,-1);
\vertex (0,0);
\vertex (0,1);
\vertex (0,2);
\vertex (0,3);
\vertex (0,4);
\vertex (0,5);
\vertex (0,6);
\vertex (1,3);
\draw (0,6) node {\llap{$e_1=({+}{-}000000;0,0)$}\kern15pt};
\draw (0,5) node {\llap{$e_2=(0{+}{-}00000;0,0)$}\kern15pt};
\draw (0,4) node {\llap{$e_3=(00{+}{-}0000;0,0)$}\kern15pt};
\draw (0,3) node {\llap{$e_4=(000{+}{-}000;0,0)$}\kern15pt};
\draw (0,2) node {\llap{$e_5=(0000{+}{-}00;0,0)$}\kern15pt};
\draw (0,1) node {\llap{$e_6=(00000{+}{-}0;0,0)$}\kern15pt};
\draw (0,0) node {\llap{$e_7=(000000{+}{-};1,0)$}\kern15pt};
\draw (0,-1) node {\llap{$e_9'=(00000000;-1,1)$}\kern15pt};
\draw (0,-2) node {\llap{$e_{10}'=(\frac-2\frac-2\frac-2\frac-2\frac-2\frac-2\frac-2\frac-2;1,0)$}\kern15pt};
\draw (0,-3) node {\llap{$\nu=(\frac+2\frac+2\frac+2\frac+2\frac+2\frac+2\frac+2\frac+2;0,0)$}\kern15pt};
\draw (1,3) node {\kern15pt\rlap{$(\frac-2\frac-2\frac-2\frac-2\frac+2\frac+2\frac+2\frac+2;0,0)=e_8$}};
\end{tikzpicture}%
\caption{Simple roots for $\Oup(\Lambda)_\nubar$;
see lemma~\ref{lem-simple-roots-for-nu-bar-stabilizer}.}
\label{fig-simple-roots-for-nu-bar-stabilizer}
\end{figure}

\begin{lemma}
\label{lem-simple-roots-for-nu-bar-stabilizer}
  $\Oup(\Lambda)_\nubar$ is the Coxeter group with simple roots
   in figure~\ref{fig-simple-roots-for-nu-bar-stabilizer}.
\end{lemma}

\begin{proof}
  First, the affine subgroup $\Etilde_7\Atilde_1$ generated by the
  reflections in all the roots except $e_9'$ is the full
  $\Oup(\Lambda)_\nubar$-stabilizer of the null vector
  $\rho=(0;1,0)$.  The proof is just like lemma~\ref{lem-null-vector-stabilizer-237-case}, and one can
  even use the same formula \eqref{eq-translations-in-E9} for the translations.  The main
  difference is that only half of the translations preserve $\nubar$.
  The ones that do correspond to the sublattice $E_7\oplus A_1$ of $E_8$.

 The Coxeter group lies in
 $\Oup(\Lambda)_\nubar$ because  every simple root has even inner
 product with $\nu$.  Cusp-counting shows that the Coxeter
 group has two orbits on null vectors, corresponding to the diagrams
 $\Etilde_8$ and $\Etilde_7\Atilde_1$.  And $\Oup(\Lambda)_\nubar$ has
 at least two orbits on null vectors, since it has two orbits on
 isotropic vectors in $V$  (those orthogonal to $\nubar$ and
 those not).  It follows that $\Oup(\Lambda)_\nubar$
 and this reflection subgroup have the same orbits on null vectors.
 The equality of these groups follows because their subgroups stabilizing $\rho$ are
 equal.
\end{proof}

\begin{lemma}
  \label{lem-nef-cone-contains-future-cone-of-nu-perp}
  The cone $\nef(X)$ contains the entire future cone of $\nu^\perp$.
\end{lemma}

  Recall that $\nef(X)$ is the Weyl chamber for the group
  $\Wnodal\sset\Orth(\Lambda)$ generated by the reflections in the nodal
  roots.  (These reflections do not arise from symmetries of $X$, but
  they are isometries of $\Lambda$.)

  \begin{proof}
  Because $X$ has a single nodality class $\nubar$, every nodal root
  represents it.  In particular, if $\nu'$ is any nodal root then
  $\Lambda$ contains $(\nu-\nu')/2$, and \hbox{$\nu\cdot\nu'$} is even (since
  $\nubar$ is isotropic).  We cannot have $\nu\cdot\nu'=0$, because
  then $(\nu-\nu')/2$ would have norm~$-1$, which is impossible in the
  even lattice $\Lambda$.  For $\nu'\neq\nu$ this implies $\nu\cdot\nu'\geq2$, so
  their orthogonal complements do not intersect in
  hyperbolic space $H^9$.  Since the orthogonal complements of the
  nodal roots bound $\nef(X)$,  their future cones lie
  in it.  In particular, $\nef(X)$ contains the future cone of $\nu^\perp$.
\end{proof}

\begin{lemma}
\label{lem-nef-cone-contains-246-translates-of-C}
  All the $W_{246}$-translates of~$C$ lie in $\nef(X)$.
\end{lemma}

\begin{proof}
  Every nodal root has the same
  nodality class $\nubar$, so all their reflections preserve it.  So the
  mirrors of $\Wnodal$ are among the mirrors of
  $\Oup(\Lambda)_\nubar$.  It follows that every chamber for
  $\Wnodal$
  is a union of chambers for
  $\Oup(\Lambda)_\nubar$.
  In particular, $\nef(X)$ is such a union.
  
  From this and lemma~\ref{lem-nef-cone-contains-future-cone-of-nu-perp} it follows that
  $\nef(X)$ contains every chamber of $\Oup(\Lambda)_\nubar$ that lies
  on the positive side of $\nu^\perp$ and has one of its facets lying
  in $\nu^\perp$.  In particular, $\nef(X)$ contains $C$.

  Now consider the following Coxeter group $W'$ lying between
  $\Wnodal$ and $\Oup(\Lambda)_\nubar$: the one generated by the
  reflections in all the roots of $\Lambda$ lying over $\nubar$.
  Writing $C'$ for its chamber containing $C$, we obviously have
  $C'\sset\nef(X)$.  Therefore it suffices to show that $C'$ contains
  the $W_{246}$-translates of $C$.  The advantage of $W'$ over
  $\Wnodal$ is that $W_{246}$ visibly normalizes it, hence permutes
  its chambers.  Suppose $r$ is any simple root from
  figure~\ref{fig-simple-roots-for-nu-bar-stabilizer} other than
  $\nu$.  Then a generic point of $r^\perp$ is orthogonal to no roots
  except $r$, hence to no roots of $W'$.  So $r$'s reflection
  preserves an interior point of $C'$, hence $C'$ itself.  It follows
  that $W_{246}$ preserves~$C'$, so $C'$ contains the
  $W_{246}$-translates of $C$, as desired.  (Parts \eqref{item-nef-is-union-of-W246-translates} and
  \eqref{item-transitivity-on-facets} of theorem~\ref{thm-1-nodal-case} imply $W'=\Wnodal$.)
\end{proof}

Next we will describe some symmetries of $X$, called Geiser, Bertini
and Kantor involutions.  They are defined in terms of nef
classes and nodal curves with certain properties; see the proof of
theorem~5 in \cite{Dolgachev-introduction-to-Enriques-surfaces} for
the details of
their construction.  We will describe their actions on $\Lambda$.  Their
actions on $V$ follow easily: Geiser and Bertini involutions act
trivially and Kantor involutions act by the transvection in
$\nubar$.

Geiser involutions: suppose $E_1$, $E_2$ are nef divisors with
$E_1^2=E_2^2=0$ and $E_1\cdot E_2=1$, such that $E_1+E_2$ is ample.
Then the linear system $|2E_1+2E_2|$ realizes $X$ as a $2$-fold
branched cover of the unique $4$-nodal quartic del Pezzo surface.  (It
can be defined in $\P^4$ by $0=x_0x_1 +x_2^2 = x_3x_4 +x_2^2$.)  The
deck transformation of this covering is an automorphism $G$ of $X$,
called a Geiser involution.  Its action on $\Lambda$ can be described
as follows.  The classes of $E_1,E_2$ in $\Lambda$ span a summand
isometric to
$U\iso\bigl(\begin{smallmatrix}0&1\\1&0\end{smallmatrix}\bigr)$, and
  $G$ acts by the negation map of this summand's orthogonal
  complement.

Bertini involutions: now suppose $E$ is a nef divisor with $E^2=0$,
and that $R$ is a nodal curve having intersection number~$1$ with it.
Then the linear system $|4E+2R|$ realizes $X$ as a $2$-fold branched
cover of a degenerate form of the previous paragraph's del Pezzo
surface.  (Its equations in $\P^4$ are $0=x_0x_1 +x_2^2= x_3x_4
+x_0^2$.)  The deck transformation is an automorphism $B$ of $X$,
called a Bertini involution.  Its action on $\Lambda$ can be described
as follows.  The classes of $E$ and $R$ span a summand $U$ of
$\Lambda$, and $B$ acts by negating its orthogonal complement.  
(This resembles the Geiser involution case if one
thinks of $E_1$ as $E$ and $E_2$ as the image of $E$ under reflection
in $R$.  The differences are that $E_2$ is not nef and
$R\cdot(E_1+E_2)=0$.  In particular, $|2E_1+2E_2|$ collapses $R$ to a
point.)

Kantor involutions: now suppose $E_1$ and $E_2$ are nef divisors with
$E_1^2=E_2^2=0$ and $E_1\cdot E_2=1$, and that $R$ is a nodal curve
disjoint from them.  Then the linear system $|2E_1+2E_2-R|$ realizes
$X$ as a $2$-fold branched cover of the Cayley cubic (the unique cubic
surface with four $A_1$ singularities).  The deck transformation of
this covering is an automorphism $K$ of $X$, called a Kantor
involution.  Its action on $\Lambda$ can be described as follows.  The
classes of $E_1$ and $E_2$ in $\Lambda$ generate a summand isometric
to $U$, and $K$ acts as the composition of the negation map on its
orthogonal complement and the reflection in the nodal root
corresponding to $R$.

\begin{remarks}
In section~\ref{sec-preparation-for-1-nodal} we introduced some isometries of $\Lambda$ that we
called Bertini and Kantor involutions.  As the language suggests, they
are special cases of the Bertini and Kantor involutions given here:

A Bertini involution in section~\ref{sec-preparation-for-1-nodal} meant an involution of
$\Lambda$ whose
negated lattice is isometric to $E_8$ and orthogonal to $\nu$.  First
we give an example of such an involution arising from the construction
above.  Consider the sublattice $L$ of $\Lambda$ spanned by the roots
of the $E_8$ subdiagram of figure~\ref{fig-simple-roots-for-nu-bar-stabilizer}.
Write $B$ for its Bertini involution in the sense of section~\ref{sec-preparation-for-1-nodal}:
it negates $L$ and fixes $L^\perp$ pointwise.
Computation shows that $L^\perp$
is spanned by $\nu$ and the null vector $E=(-1,0,0,0,0,0,0,-1;\discretionary{}{}{}1,1)$,
which have inner product~$1$.  It is easy to check that $E$ has inner
product${}\geq0$ with the simple roots in figure~\ref{item-transitivity-on-facets}, so it lies in
$C$ and hence is nef (lemma~\ref{lem-nef-cone-contains-246-translates-of-C}).  The Bertini involution
constructed above, using $E$ and $R=\nu$, is exactly $B$.

Now consider any Bertini involution $B'$ in the sense of
section~\ref{sec-preparation-for-1-nodal}, and write $L'\iso E_8$ for its negated lattice.  Any
two copies of the $E_8$ lattice in $\nu^\perp$ are equivalent under
isometries of $\nu^\perp$, because each is a direct summand (being
unimodular), whose $1$-dimensional complement in $\nu^\perp$ 
has the same determinant as $\nu^\perp$, namely~$2$.  Therefore some
$g\in\Oup(\nu^\perp)$ sends $L$ to $L'$.  Since
$\Oup(\nu^\perp)=W_{245}$ (theorem~\ref{thm-Vinberg-245}) and $W_{245}$ preseves
$\nef(X)$ (lemma~\ref{lem-nef-cone-contains-246-translates-of-C}), $g(E)$ is also nef.   The Bertini
involution constructed above, using $g(E)$ and $R=\nu$, is~$B'$.

Finally, suppose $K$ is a Kantor involution in the sense of section~\ref{sec-preparation-for-1-nodal},
so its negated lattice is the first summand of some decomposition
$\nu^\perp=E_7\oplus U$.  We take $E_1$, $E_2$ to be 
null vectors spanning the $U$ summand.  By lemma~\ref{lem-nef-cone-contains-future-cone-of-nu-perp} they are nef.
Then the Kantor involution constructed above from $E_1$ and $E_2$ is~$K$.
\end{remarks}

\begin{proof}[Proof of theorem~\ref{thm-1-nodal-case}]
The main step is to prove \eqref{item-Gamma-contains-W-bar-246(2)}.  For this we reuse the strategy of
theorem~\ref{thm-unnodal-case}.  We think of $W_{246}$ as the ``large'' group, and for
the ``small'' subgroup $S$ we take the subgroup generated by the
Geiser, Bertini and Kantor involutions of $X$ that lie in $W_{246}$.
(In fact these are all the Geiser, Bertini and Kantor involutions, but
we have not proven this.)
To relate these groups we define the ``medium'' group $M$ as the
subgroup generated by $S$ and $W_{245}$.  Obviously $M$ normalizes~$S$.

Theorem~\ref{thm-Kantor-and-Bertini-inside-245} says that $S\cap W_{245}$ contains the subgroup of
$W_{245}$ that acts trivially on $\nubar^\perp$.  Theorem~\ref{thm-Vinberg-245} says that
$W_{245}=\Oup(\Lambda)_\nu$.  And the image of $W_{245}$ in $\Orth(V)$
acts on $\nubar^\perp$ by its full isometry group.  Therefore the
action of $M$ on $V$ identifies $M/S$ with $\Orth(\nubar^\perp)=\Orth_9(2)$.

Now we claim that $M$ is all of $W_{246}$.  It contains $W_{245}$ by
definition, so it suffices to show that $M$ contains the reflection in
$e_{10}'$.  Observe that $e_2,\dots,e_8,e_{10}'$ span a root lattice
$E_7A_1$.  Its saturation is strictly larger, hence a copy of $E_8$,
because $\Lambda$ contains $(e_9-e_{10}'+e_7+e_5)/2$.  We will show in
the next paragraph that the negation map $G$ of this $E_8$ summand of
$\Lambda$ is a Geiser involution.  It lies in $W_{246}$ since it is
the product of the reflection in $e_{10}'$ and the central involution
of $W(E_7)$.  Therefore $G\in S\sset M$.  Since $M$ contains the
central involution of $W(E_7)$, the same decomposition of $G$ shows
that $M$ contains the reflection in $e_{10}'$.
Modulo the fact that $G$ is a Geiser involution,
this
completes the proof that $M=W_{246}$.  From our understanding of
$M/S$ it follows that $S$ is exactly the subgroup of $W_{246}$ that
acts trivially on $\nubar^\perp\sset V$.  This is definition of
$\Wbar_{\!246}(2)$, proving~\eqref{item-Gamma-contains-W-bar-246(2)}.

We must still show that $G$ is a Geiser involution.  To do this we
seek $E_1,E_2\in\nef(X)$ with $E_1^2=E_2^2=0$ and $E_1\cdot E_2=1$ and
$E_1+E_2$ ample, which span a copy of $U$ orthogonal to
$e_2,\dots,e_8,e_{10}'$.  Now, the orthogonal complement of
$e_2,\dots,e_8,e_{10}'$ has signature $(1,1)$, so it has only two
isotropic lines.  This determines $E_1$ and $E_2$  up to scaling.  In fact they are obvious from the
Dynkin diagram: they must be the null vectors representing the
$\Etilde_7\Atilde_1$ and $\Etilde_8$ cusps.  We already know that the
first is $(0,\dots,0;1,0)$, and one checks that the second is
$(-1,0,\dots,0,-1;1,1)$.     They have
inner product~$1$, hence span a copy of $U$.  They are nef because
they lie in $C$, so their  sum is nef too.  This sum also has odd inner product 
with $\nu$, hence nonzero inner product with every nodal root (since
all nodal roots are congruent mod~$2$).  Therefore the sum is ample,
so $G$ is a Geiser involution.
This completes the proof that
$M=W_{246}$ and hence the proof of~\eqref{item-Gamma-contains-W-bar-246(2)}.

Next we prove the rest.  The main point is that $W_{246,\nu}/S_\nu$
maps isomorphically to $W_{246}/S$, since $W_{246,\nu}=W_{245}$ acts
on $\nubar^\perp$ as $\Orth_9(2)$.  It follows that the $S$-orbit of
$\nu$ coincides with the $W_{246}$-orbit of $\nu$.  This shows
simultaneously that every facet of $\cup_{w\in W_{246}}w(C)$ is the
orthogonal complement of a nodal root, and that $S$ (hence $\Gamma$)
acts transitively on them.  These are claims
\eqref{item-nef-is-union-of-W246-translates} and
\eqref{item-transitivity-on-facets} of the theorem.
Claim \eqref{item-transitivity-on-nodal-curves} follows from
\eqref{item-transitivity-on-facets} and the bijection between nodal roots and
nodal curves.

For claim~\eqref{item-W246-is-full-stabilizer}, recall that every
nodal root lies over $\nubar$.  So the full $\Oup(\Lambda)$-stabilizer
of $\nef(X)$ must preserve $\nubar$.  Since $C$ is a fundamental
domain for $\Oup(\Lambda)_\nubar$ by
lemma~\ref{lem-simple-roots-for-nu-bar-stabilizer}, the stabilizer of
$\nef(X)$ is exactly the subgroup of $\Oup(\Lambda)_\nubar$ which
preserves $\nef(X)$. This is obviously $W_{246}$.
\end{proof}


\begin{thebibliography}{99}

\bibitem{ATLAS} Conway, J. H., Curtis, R. T., Norton, S. P., Parker,
  R. A., Wilson, R. A., {\it Atlas of finite groups.} Oxford
  University Press 1985.

\bibitem{Barth-Peters}
  Barth, W. and Peters, C., Automorphisms of Enriques surfaces, {\it
    Invent. Math.} {\bf 73} (1983) 383--411.

\bibitem{Bourbaki}
Bourbaki, Nicolas, {\it \'El\'ements de
  math\'ematique. Groupes et alg\`ebres de Lie. Chapitres 4, 5 et 6,}
  Masson, Paris, 1981. English translation: Lie groups and Lie
  algebras. Chapters 4–6, Springer-Verlag, Berlin, 2002
  
\bibitem{Coble}
  Coble, A., The ten nodes of the rational sextic and of the Cayley
  symmetroid, {\it American J. Math.} {\bf 41} (1919) 243--265.
  
\bibitem{Conway-Sloane}
  Conway, J. H. and Sloane, N. J. A., {\it Sphere Packings, Lattices and
  Groups.} Grundlehren der
  Mathematischen Wissenschaften {\bf 290}. Springer-Verlag, New York,
  1999.

\bibitem{CD-old}
  Cossec, F. and Dolgachev, I.,
{\it Enriques Surfaces. I.}
Progress in Mathematics {\bf 76}, Birkh\"auser, Boston, 1989.

\bibitem{CD-new} Cossec--Dolgachev, {\it Enriques Surfaces}, to
  appear. (Complete revision of \cite{CD-old})

\bibitem{Cossec-Dolgachev-published} 
Cossec, F. and Dolgachev, I., On automorphisms of nodal Enriques surfaces,
{\it Bull. Amer. Math. Soc. (N.S.)} {\bf 12} (1985) 247--249.

\bibitem{Dolgachev-introduction-to-Enriques-surfaces} Dolgachev, I., A
  brief introduction to Enriques surfaces, arxiv:1412.7744.
    
\bibitem{Vinberg-E10} Vinberg, E., Some arithmetical discrete groups
  in Lobachevskii spaces, in {\it Discrete Subgroups of Lie Groups and
    Applications to Moduli}, Oxford U.\ Press 1975, 323--348.

\end{thebibliography}
\end{document}